\documentclass[12pt,a4paper,reqno]{amsart}

\usepackage{anysize}
\usepackage[latin1]{inputenc}
\usepackage{amsmath}
\usepackage{amssymb}
\usepackage{array}

\newtheorem{theorem}{Theorem}[section]

\newtheorem{lemma}[theorem]{Lemma}

\newcommand{\hpg}[5]{{}_{#1}\mbox{\rm F}_{\!#2}\! \left(\left.{#3 \atop #4}\right| #5 \right)}

\begin{document}

\title{More hypergeometric identities related to Ramanujan-type series}
\author{Jes\'{u}s Guillera}
\address{Zaragoza, (Spain)}
\email{jguillera@gmail.com}
\date{}

\dedicatory{Dedicated to Jonathan Borwein and Doron Zeilberger \\
on the occasion of their $61^{st}$ and $62^{nd}$ birthday respectively}

\begin{abstract}
We  find new hypergeometric identities which, in a certain aspect, are stron-ger than others of the same style found by the author in a previous paper. The identities in Section \ref{section-pi} are related to some Ramanujan-type series for $1/\pi$. We derive them by using WZ-pairs associated to some interesting formulas by Wenchang Chu. The identities we prove in Section \ref{section-pi2} are of the same style but related to Ramanujan-like series for $1/\pi^2$.
\end{abstract}

\maketitle

\section{Introduction}\label{introduc}

\subsection{Ramanujan series}

In 1914 S. Ramanujan gave $17$ series for $1/\pi$ which are of the following form:
\begin{equation}\label{rama-series}
\sum_{n=0}^{\infty} z^n \frac{ \left( \frac{1}{2} \right)_n (s)_n (1-s)_n}{ (1)_n^3} (a+bn)= \frac{1}{\pi}, \qquad (x)_k=\frac{\Gamma(x+k)}{\Gamma(x)},
\end{equation}
where $s \in \{1/2,1/4,1/3,1/6 \}$ and the parameters $z$, $a$, $b$ are algebraic numbers.
One example is
\begin{equation}\label{rama42}
\sum_{n=0}^{\infty} \frac{1}{2^{6n}} \; \frac{\left( \frac12 \right)_n^3}{(1)_n^3} \; (42n+5)=\frac{16}{\pi}.
\end{equation}
See Ramanujan's paper \cite{Ra} and \cite[pp. 352-354]{Be}. In $1987$ the Borwein brothers gave rigorous proofs of the $17$ Ramanujan series \cite{BoBo}.
In honor of Ramanujan, series of the form (\ref{rama-series}) are now known as Ramanujan-type series for $1/\pi$. An excellent survey of such results is provided in \cite{BaBeCh}, and in \cite{BeBoBo} there are many good articles concerned with the number $\pi$, some of which are related to Ramanujan-type series.

\subsection{Ramanujan series extended}

We consider the following extensions with the variable $x$ of the Ramanujan-type series:
\begin{equation}\label{rama-pos-ext}
\sum_{n=0}^{\infty} z^{n+x} \frac{ \left( \frac12 \right)_{n+x} (s)_{n+x} (1-s)_{n+x} }{ (1)_{n+x}^3 } (a+b(n+x)),
\end{equation}
and
\begin{equation}\label{rama-alt-ext}
\sum_{n=0}^{\infty} (-1)^n z^{n+x} \frac{ \left( \frac12 \right)_{n+x} (s)_{n+x} (1-s)_{n+x}}{ (1)_{n+x}^3} (a+b(n+x)),
\end{equation}
for $0<z<1$. In \cite[Sect. 2]{Gu-hyperiden}, we used the WZ-method \cite{Pe} to derive many hypergeometric identities which can be written in the form
\begin{equation}\label{form1-pos}
\sum_{n=0}^{\infty} g(n+x)=t(x)+x^2 \sum_{n=0}^{\infty} f(x,n)
\end{equation}
or in the form
\begin{equation}\label{form1-neg}
\sum_{n=0}^{\infty} (-1)^n g(n+x)=t(x)+x^2 \sum_{n=0}^{\infty} f(x,n),
\end{equation}
where $g(n)$ is hypergeometric, $f(k,n)$ is hypergeometric in both variables, $t(x)$ is a simple trigonometric function such that $t(0)=1/\pi$, and
$\sum_{n=0}^{\infty} g(n)=t(0)$ is a Ramanujan-type series for $1/\pi$.
From (\ref{form1-pos}) or (\ref{form1-neg}), we can calculate the coefficient of $x^2$ by observing that
\[ \sum_{n=0}^{\infty} f(x,n)=\sum_{n=0}^{\infty} f(0,n)+\mathcal{O}(x). \]
In Section \ref{section-pi} of this paper, we use WZ-pairs to find stronger versions of identities (\ref{form1-pos}) and (\ref{form1-neg}). They are of the form
\begin{equation}\label{form2-pos}
\sum_{n=0}^{\infty} g(n+x)=t(x)+x^3 \sum_{n=0}^{\infty} f(x,n),
\end{equation}
and
\begin{equation}\label{form2-alt}
\sum_{n=0}^{\infty} (-1)^n g(n+x)=t(x)+x^3 \sum_{n=0}^{\infty} f(x,n).
\end{equation}
Hence $t(x)$ easily gives the first two terms of the series expansions of the corresponding extended series \cite{Gu-matrix} of the form
\[ \sum_{n=0}^{\infty} g(n+x) \quad \text{and} \quad \sum_{n=0}^{\infty} (-1)^n g(n+x). \]
In addition, we can obtain the coefficient of $x^3$ by observing that
\[ \sum_{n=0}^{\infty} f(x,n)=\sum_{n=0}^{\infty} f(0,n)+\mathcal{O}(x). \]
This provides a substantial improvement to our results from \cite{Gu-hyperiden}. As an example, compare the old identity
\begin{multline}\label{old}
\frac{1}{16} \sum_{n=0}^{\infty} \frac{1}{2^{6(n+x)}} \frac{ \left(\frac{1}{2}\right)_{n+x}^3}{(1)_{n+x}^3}(42(n+x)+5) \\ =
\frac{1}{\pi} \frac{1}{\cos^2\pi x} + \frac{1}{64^x} \frac{\left( \frac12 \right)_x^3}{(1)_x^3} \frac{8 x^2}{2x-1} \sum_{n=0}^{\infty}  \frac{\left(\frac12+x\right)_n^2}{(1+2x)_n \left( \frac32 - x \right)_n},
\end{multline}
with our new identity
\begin{multline}\label{new}
\frac{1}{16} \sum_{n=0}^{\infty} \frac{1}{2^{6(n+x)}}
\frac{\left(\frac{1}{2}\right)_{n+x}^3}{(1)_{n+x}^3}[42(n+x)+5] \\
=\frac{1}{\pi} \frac{\cos 3 \pi x}{\cos^3 \pi x} + \frac{1}{64^x} \frac{\left(\frac{1}{2} \right)_x^3} {(1)_x^3}
\frac{64x^3}{(1-2x)^3} \sum_{n=0}^{\infty} \frac{\left( \frac{1}{2}+x \right)_n^3}{\left( \frac{3}{2}-x \right)_n^3} (-1)^n(2n+1).
\end{multline}
Equation (\ref{old}) implies the expansion
\[
\frac{1}{16} \sum_{n=0}^{\infty} \frac{1}{2^{6(n+x)}} \frac{ \left(\frac{1}{2}\right)_{n+x}^3}{(1)_{n+x}^3}(42(n+x)+5)=
\frac{1}{\pi}-3 \pi x^2 + \mathcal{O}(x^3),
\]
but  does not easily yield the coefficient of $x^3$. However, from (\ref{new}) we easily deduce
\[
\frac{1}{16} \sum_{n=0}^{\infty} \frac{1}{2^{6(n+x)}} \frac{ \left(\frac{1}{2}\right)_{n+x}^3}{(1)_{n+x}^3}(42(n+x)+5)=
\frac{1}{\pi}-3 \pi x^2 + 64 G x^3 + \mathcal{O}(x^4),
\]
where $G$ is Catalan's constant. In Section \ref{section-pi2} we obtain analogous hypergeometric identities related to the two infinite series:
\begin{equation}\label{pi2-1}
\sum_{n=0}^{\infty}  \frac{\left( \frac12 \right)_n^3 \left( \frac14 \right)_n
\left( \frac34 \right)_n}{(1)_n^5} \frac{1}{2^{4n}} (120n^2+34n+3) =\frac{32}{\pi^2},
\end{equation}
and
\begin{equation}\label{pi2-2}
\sum_{n=0}^{\infty}  \frac{\left(\frac12 \right)_n^3 \left( \frac13 \right)_n \left( \frac23 \right)_n}{(1)_n^5} \left( \frac34 \right)^{3n} (74n^2+27n+3)=\frac{48}{\pi^2}.
\end{equation}
Identities (\ref{pi2-1}) and (\ref{pi2-2}), and two additional cases, were discovered and proved by the author via the Wilf-Zeilberger (WZ) method in 2002, 2003 and 2010 (see \cite{guilleraAAMwz}, \cite{guilleraRJgen} and \cite{Gu-wznew}). The author also conjectured six moreover formulas for $1/\pi^2$ (see \cite{Gu-newkind} and \cite{Gu6}). One additional case was conjectured by Gert Almkvist and the author in \cite{AlGu}. Following the publication of \cite{guilleraAAMwz}, Boris Gourevitch discovered a similar series for $1/\pi^3$, using experimental integer relation algorithms (see \cite{Gu-newkind}), and most recently Jim Cullen gave a formula for $1/\pi^4$ (see \cite{Zu}).

\subsection{The WZ-method}

A function $A(n,k)$ is hypergeometric in two variables if the quotients
\[
\frac{A(n+1,k)}{A(n,k)} \quad {\rm and} \quad \frac{A(n,k+1)}{A(n,k)}
\]
are rational functions in $n$ and $k$, respectively. A pair of hypergeometric functions, $F(n,k)$ and $G(n,k)$, is said to be a WZ pair \cite[Chapt. 7]{Pe} if
\begin{equation}\label{pro-WZ-pair}
F(n+1,k)-F(n,k)=G(n,k+1)-G(n,k).
\end{equation}
Wilf and Zeilberger \cite{wilf} proved that in this case, there exists a rational function $C(n,k)$ such that
\begin{equation}\label{certificado}
G(n,k)=C(n,k)F(n,k).
\end{equation}
The rational function $C(n,k)$ is called the certificate of the pair $(F,G)$. To discover WZ-pairs, we use Zeilberger's Maple package EKHAD \cite[Appendix A]{Pe}. If EKHAD certifies a function, we have found a WZ-pair! The usual approach is to sum (\ref{pro-WZ-pair}) for $n \geq 0$, to get
\begin{equation}\label{sumas-wz}
\sum_{n=0}^{\infty} G(n,k) - \sum_{n=0}^{\infty} G(n,k+1) = -F(0,k) + \lim_{n \to \infty} F(n,k).
\end{equation}
In this paper we use an alternative method explained in the next section.

\section{A strategy based on the WZ-method}

The following two theorems establish the main property of the function $t(x)$:
\begin{theorem}\label{teor1}
Let $F(n,k)$ and $G(n,k)$ be a WZ-pair \cite{Pe}. If $G(n,0)$ and $G(k,n)$ vanish as $n \to \infty$,
then the function $t(x)$ defined by
\begin{equation}
\sum_{n=0}^{\infty} G(n+x,0)-\sum_{n=0}^{\infty} F(x,n)=t(x)
\end{equation}
satisfies $t(x+1)=t(x)$, for all complex numbers $x$.
\end{theorem}
\begin{proof}
Let $F(n,k)$, $G(n,k)$ be a WZ-pair and define
\[ H(n,k)=G(n+k,0), \qquad G_1(n,k)=F(k,n), \qquad  F_1(n,k)=G(k,n). \]
Obviously we have
\[ H(n,k+1)-H(n,k)=H(n+1,k)-H(n,k) \]
and
\[ G_1(n,k+1)-G_1(n,k)=F_1(n+1,k)-F_1(n,k). \]
We now define
\[ F_2(n,k)=H(n,k)-F_1(n,k), \qquad G_2(n,k)=H(n,k)-G_1(n,k). \]
Thus
\[ G_2(n,k+1)-G_2(n,k)=F_2(n+1,k)-F_2(n,k). \]
Summing for $n \geq 0$, we obtain
\begin{align}
\sum_{n=0}^{\infty} \big[G_2(n,k+1)-G_2(n,k)\big] &= -F_2(0,k)+\lim_{n \to \infty} F_2(n,k) \nonumber \\ &=-G(k,0)+G(k,0)+\lim_{n \to \infty} G(n+k,0) - \lim_{n \to \infty} G(k,n) \nonumber \\
&=- \lim_{n \to \infty} G(k,n)=0. \nonumber
\end{align}
Hence, we arrive at the following result which completes the proof:
\[ \sum_{n=0}^{\infty} G_2(n,k)=t(k), \]
where $t(k)$ is a periodic function of period one, that is $t(k+1)=t(k)$.
\end{proof}

\begin{theorem}\label{teor2}
Let $F(n,k)$ and $G(n,k)$ be a WZ-pair involving a factor $(-1)^n$; that is $F(n,k)=(-1)^n \widehat{F}(n,k)$, $G(n,k)=(-1)^n \widehat{G}(n,k)$. If $G(n,0)$ and $G(k,n)$ vanish as $n \to \infty$, then the function $t(x)$ defined by
\begin{equation}
\sum_{n=0}^{\infty} (-1)^n \widehat{G}(n+x,0)-\sum_{n=0}^{\infty} \widehat{F}(x,n)=t(x),
\end{equation}
has the property that $t(x+1)=-t(x)$, for all complex numbers $x$.
\end{theorem}
\begin{proof}
Just write $G(n,k)=\cos(\pi n) \widehat{G}(n,k)$ and apply the first theorem.
\end{proof}

We also use the following theorem:
\begin{theorem}\label{Fou}
Let $h(z)$ be an entire function which is periodic of period $1$ and let
\[ h(z)=\sum_{n=-\infty}^{+\infty} a_n e^{2 \pi i n z}, \]
be its Fourier series. If $h(z)=\mathcal{O}(e^{c \pi |{\rm Im}(z)|})$ for $|{\rm Im}(z)|$ sufficiently large, where  $c \geq 0$ is a constant, then $c-2|n|<0 \, \Rightarrow a_n=0$.
\end{theorem}
\begin{proof}
Since $h(z)$ has period $1$, its Fourier coefficients are given by
\[ a_n=\int_{a+ib}^{1+a+ib} h(z) e^{-2 \pi i n z} dz, \]
where $a$ and $b$ are arbitrary real numbers. Then for $n<0$ let $b \to+\infty$ and for $n \geq 0$ let $b \to -\infty$ to conclude $c-2|n|<0 \, \Rightarrow a_n=0$.
\end{proof}

Before finding new identities, we apply our new theorems to complete the proof of an identity conjectured in \cite{Gu-hyperiden}. The same idea also applies to other identities in that paper.

\subsection{Example}\label{ejem}

\begin{multline}
\frac{\sqrt 2}{4} \sum_{n=0}^{\infty} \frac{(-1)^n}{2^{3(n+x)}}
\frac{\left(\frac{1}{2}\right)_{n+x}^3}{(1)_{n+x}^3}[6(n+x)+1] \\
=\frac{1}{\pi} \frac{1}{\cos \pi x} + \frac{4 \sqrt 2}{8^x} \, \frac{x^2}{2x-1} \frac{\left(\frac{1}{2} \right)_x^3} {(1)_x^3} \sum_{n=0}^{\infty}  \frac{1}{2^n} \cdot \frac{\left( x+\frac{1}{2} \right)_n^2} {(x+1)_n \left( \frac{3}{2}-x \right)_n}.
\end{multline}
\begin{proof}
In \cite{Gu-hyperiden} we guessed that the function $s(x)=s_1(x)+s_2(x)$, where
\[
s_1(x)=\frac{\sqrt{2}}{4} \frac{1}{8^x} \frac{\left( \frac12 \right)_x^3}{(1)_x^3} \sum_{n=0}^{\infty} \frac{(-1)^n}{2^{3n}}
\frac{\left(\frac{1}{2}+x\right)_{n}^3}{(1+x)_{n}^3}[6(n+x)+1],
\]
and
\[
s_2(x)=-\frac{4 \sqrt 2}{8^x} \, \frac{x^2}{2x-1} \frac{\left(\frac{1}{2} \right)_x^3} {(1)_x^3} \sum_{n=0}^{\infty}  \frac{1}{2^n} \cdot \frac{\left( x+\frac{1}{2} \right)_n^2} {(x+1)_n \left( \frac{3}{2}-x \right)_n},
\]
was equal to the simple function $t(x)=1/(\pi \cos \pi x)$. Here we prove that it is true. Define the function $h(x)=s(x)\cos \pi x$. We first prove that $h(x)$ has no poles. By Th. \ref{teor2} we see that the function $h(x)$ has the property $h(x+1)=h(x)$, thus we only need to check that $h(x)$ has no poles in the band $0 \leq {\rm Re}(x) < 1$. The only possible pole in this band is at $x=1/2$, but taking the limit of $h(x)$ as $x \to 1/2$, we clearly see that it is not a pole. Hence $h(x)$ is holomorphic and periodic and this implies that it has a Fourier series expansion. By the Weiertrass $M$ test, as $|{\rm Im}(x)| \to \infty$, we have
\[
\sum_{n=0}^{\infty} \frac{(-1)^n}{2^{3n}} \frac{\left(\frac{1}{2}+x\right)_{n}^3}{(1+x)_{n}^3}(6n+1) \to
6 \sum_{n=0}^{\infty} \frac{(-1)^n}{8^n} \frac{\left(\frac12 \right)_{n}^3}{(1)_{n}^3}(6n+1)
\]
and
\[
\sum_{n=0}^{\infty} \frac{(-1)^n}{2^{3n}} \frac{\left(\frac{1}{2}+x\right)_{n}^3}{(1+x)_{n}^3} \to
6 \sum_{n=0}^{\infty} \frac{(-1)^n}{8^n} \frac{\left(\frac12 \right)_{n}^3}{(1)_{n}^3},
\]
which are finite. Hence
\[ \sum_{n=0}^{\infty} \frac{(-1)^n}{2^{3n}} \frac{\left(\frac{1}{2}+x\right)_{n}^3}{(1+x)_{n}^3}[6(n+x)+1]= \mathcal{O}(|{\rm Im}(x)|). \]
On the other hand, as $|{\rm Im}(x)| \to \infty$, we have the behaviour
\[ \frac{\left( \frac12 \right)_x^3}{(1)_x^3}=\mathcal{O}(|{\rm Im}(x)|^{-3/2}). \]
Thus, $|s_1(x)| \to 0$ for $|{\rm Im}(x)| \to \infty$, and in a similar way we see that $|s_2(x)| \to 0$. This implies that $|h(x)|=\mathcal{O}(e^{\pi |{\rm Im}(x)|})$ for $|{\rm Im}(x)|$ sufficiently large. As $h(0)$ is a Ramanujan series with the sum $1/\pi$ \cite[Table 2]{guilleraRJgen}, using Thm. \ref{Fou}, we deduce that
\[ h(x)=\frac{1}{\pi} + 0 \cdot \cos 2 \pi x + 0 \cdot \sin 2 \pi x + 0 \cdot \cos 4 \pi x + \cdots, \]
and we conclude that $s(x)=t(x)$.
\end{proof}
For the last identity in this paper (\ref{idenpi2}), we will explain how we guessed the function $t(x)$ corresponding to it (\ref{funt}).

\section{Identities related to Ramanujan-type series for $1/\pi$}\label{section-pi}

We find new hypergeometric identities using WZ-pairs associated to formulas of Wenchang Chu \cite[Cor. 2.10, 2.33, 2.21, 2.4]{chu}.  In the proofs of Identities 1, 2, 3, 4 the hypothesis
\[ \lim_{n \to \infty} G(k,n) = 0 \]
does not hold. However, with the help of Barnes integral, we can relax this hypothesis and use instead the condition
\begin{equation}\label{thin-contour}
\int_{\mathcal{C}} \widetilde{G}(t,s)ds=0,
\end{equation}
where the $\mathcal{C}$ is a very thin contour and $\widetilde{G}$ is a type of transformation of $G$ (see \cite{Gu-barnes} for details). We state this condition explicitly in the proof of Identity $4$. A related application of the Barnes integral is given in \cite{St-barnes}.

\subsection{Four new identities}

Our proof of Identity $1$ uses the following lemma:

\begin{lemma}\label{lemma1}
The following expansion is true as $x \to 1/2$
\begin{equation}\label{asymp1}
\sum_{n=0}^{\infty} \frac{\left( \frac{1}{2}+x \right)_n^2}{\left( \frac{3}{2}-x \right)_n^2} (-1)^n=
\frac{1}{2}-(\ln 2) (2x-1)+\mathcal{O}(2x-1)^2.
\end{equation}
\end{lemma}
\begin{proof}
We use the formula for the derivative of $(x)_n$ given in \cite[p. 17]{Mo-integrals}:
\[ \frac{d}{dx} (x)_n=(x)_n \sum_{j=0}^{n-1} \frac{1}{j+x}. \]
Applying it to the function
\[ f(x)=\frac{ \left( \frac{1}{2}+x \right)_n }{ \left( \frac{3}{2}-x \right)_n }, \]
we get
\[ f' \left( \frac12 \right)= 2H_n, \]
where
\[ H_n=\sum_{j=1}^{n} \frac{1}{j} \]
is the harmonic number. Hence
\begin{align}
\sum_{n=0}^{\infty} \left[ \frac{\left( \frac{1}{2}+x \right)_n}{\left( \frac{3}{2}-x \right)_n} \right]^2 z^n &=
\sum_{n=0}^{\infty} z^n + 2 \left(\sum_{n=0}^{\infty} H_n z^n \right)
(2x-1)+\mathcal{O}(2x-1)^2 \nonumber \\ &=\frac{1}{1-z}-2\frac{\ln(1-z)}{1-z} (2x-1)+\mathcal{O}(2x-1)^2. \nonumber
\end{align}
Finally we take the limit as $z \to -1$.
\end{proof}

\subsection*{Identity 1}

\begin{multline}\label{iden1}
\frac{1}{4} \sum_{n=0}^{\infty} \frac{1}{2^{2(n+x)}}
\frac{\left(\frac{1}{2}\right)_{n+x}^3}{(1)_{n+x}^3}[6(n+x)+1] \\
=\frac{1}{\pi} \frac{\cos 2 \pi x}{\cos^2 \pi x} + \frac{1}{4^x} \frac{\left(\frac{1}{2} \right)_x^3} {(1)_x^3}
\frac{16 x^3}{(1-2x)^2} \sum_{n=0}^{\infty} \frac{\left( \frac{1}{2}+x \right)_n^2}{\left( \frac{3}{2}-x \right)_n^2} (-1)^n.
\end{multline}
Equation (\ref{iden1}) leads to the expansion
\begin{equation}
\frac{1}{4} \sum_{n=0}^{\infty} \frac{1}{2^{2(n+x)}}
\frac{\left(\frac{1}{2}\right)_{n+x}^3}{(1)_{n+x}^3}[6(n+x)+1]=\frac{1}{\pi}-\pi x^2 + 16 G x^3 + \mathcal{O}(x^4),
\end{equation}
where $G$ is the Catalan's constant.

\begin{proof}
Apply Thm. \ref{teor1}, and the condition (\ref{thin-contour}), to the WZ-pair
\begin{align}
F(n,k)&=\frac { \left( \frac12 - k \right)_n^2 \left( \frac12 + k \right)_n^2}
{(1)_n^3 \left( \frac12 \right)_n} (-1)^k \frac{1}{4^n} \frac{64n^3}{(2n-2k-1)^2}, \nonumber \\
G(n,k)&=\frac { \left( \frac12 - k \right)_n^2 \left( \frac12 + k \right)_n^2}
{(1)_n^3 \left( \frac12 \right)_n} (-1)^k \frac{1}{4^n} \frac{(2n+1)(6n+1)-4k^2}{2n+1}. \nonumber
\end{align}
We see that the function $s(x)=s_1(x)+s_2(x)$, where
\[
s_1(x)=\frac{1}{4} \sum_{n=0}^{\infty} \frac{1}{2^{2(n+x)}} \frac{\left(\frac{1}{2}\right)_{n+x}^3}{(1)_{n+x}^3}[6(n+x)+1],
\]
and
\[
s_2(x)=- \frac{1}{4^x} \frac{\left(\frac{1}{2} \right)_x^3} {(1)_x^3} \frac{16x^3}{(1-2x)^2} \sum_{n=0}^{\infty} \frac{\left( \frac{1}{2}+x \right)_n^2}{\left( \frac{3}{2}-x \right)_n^2} (-1)^n,
\]
has the property $s(x+1)=s(x)$. Then we guess that $s(x)$ is equal to
\[ t(x)=\frac{1}{\pi} \frac{\cos 2 \pi x}{\cos^2 \pi x}. \]
To prove this result, define the function $h(x)=s(x)-t(x)$. We know that $h(0)=0$ and that $h(x+1)=h(x)$. Since $h(x)$ has period $1$, we can prove that it is holomorphic by checking that it has no poles in the band $0 \leq {\rm Re}(x)<1$. The only possible pole in this band is at $x=1/2$, but using (\ref{asymp1}) to evaluate the limit of the function $h(x)$ as $x \to 1/2$, shows it is not a pole. As in the example in Section \ref{ejem}, we see that $s_1(x)$ tends to $0$ as $|{\rm Im}(x)| \to \infty$. On the other hand, (and although we do not have a rigorous proof of it), numerical results clearly show that for $|{\rm Im}(x)|$ sufficiently large, we have $|s_2(x)|=\mathcal{O}(1)$. Then we deduce that $|h(x)|=\mathcal{O}(1)$. Hence, by Thm. \ref{Fou}, the Fourier expansion of $h(x)$ reduces to a constant (observe that as $|h(x)|$ is bounded, also Liouville's Theorem from complex analysis states that $h(x)$ is a constant). Since $h(0)=0$, we see that $h(x)=0$ for all complex $x$.
\end{proof}
\textbf{Remark:} We interpret all the series in this paper in terms of their analytic continuations. For example, we define the series in $s_2(x)$ in the following way:
\[
\sum_{n=0}^{\infty} \frac{\left( \frac12+x \right)_n^2}{\left( \frac32-x \right)_n^2} (-1)^n :=
\hpg32{\frac12+x, \, \frac12+x, \, 1}{\frac32-x, \, \frac32-x}{\, -1}.
\]
In fact, it is not necessary to use the sum at all, but it makes for easier reading.
\par Our proof of Identity $2$ uses the following lemma:
\begin{lemma}\label{lemma2}
The following expansion is true when $x \to 1/2$
\begin{equation}\label{asymp2}
\sum_{n=0}^{\infty} \frac{\left( \frac{1}{2}+x \right)_n^3}{\left( \frac{3}{2}-x \right)_n^3} (-1)^n (2n+1)=
-\frac{3}{2}(2x-1)+\mathcal{O}(2x-1)^3.
\end{equation}
\end{lemma}
\begin{proof}
Using symbolic calculations with Maple, and replacing $\Psi(1+n)+\gamma$ with $H_n=\sum_{j=1}^{n} \frac{1}{j}$, we see that
\begin{align}
\sum_{n=0}^{\infty} \left[ \frac{\left( \frac{1}{2}+x \right)_n}{\left( \frac{3}{2}-x \right)_n} \right]^3 (2n+1) z^n &= \sum_{n=0}^{\infty} (2n+1) z^n \nonumber + 3 \left(\sum_{n=0}^{\infty} (2n+1) H_n z^n \right)
(2x-1) \nonumber \\ &+ \frac{9}{2}  \left(\sum_{n=0}^{\infty} (2n+1) H_n^2 z^n \right) (2x-1)^2+\mathcal{O}(2x-1)^3. \nonumber
\end{align}\nonumber
Then, use the representation
\[ H_n=\int_0^1 \frac{1-(1-x)^n}{x}dx \]
and take the limit as $z \to -1$.
\end{proof}

\subsection*{Identity 2}

\begin{multline}\label{iden2}
\frac{1}{16} \sum_{n=0}^{\infty} \frac{1}{2^{6(n+x)}}
\frac{\left(\frac{1}{2}\right)_{n+x}^3}{(1)_{n+x}^3}[42(n+x)+5] \\
=\frac{1}{\pi} \frac{\cos 3 \pi x}{\cos^3 \pi x} + \frac{1}{64^x} \frac{\left(\frac{1}{2} \right)_x^3} {(1)_x^3}
\frac{64x^3}{(1-2x)^3} \sum_{n=0}^{\infty} \frac{\left( \frac{1}{2}+x \right)_n^3}{\left( \frac{3}{2}-x \right)_n^3} (-1)^n(2n+1).
\end{multline}
From (\ref{iden2}), we obtain the expansion
\begin{equation}
\frac{1}{16} \sum_{n=0}^{\infty} \frac{1}{2^{6(n+x)}}
\frac{\left(\frac{1}{2}\right)_{n+x}^3}{(1)_{n+x}^3}[42(n+x)+5]=\frac{1}{\pi}-3 \pi x^2 + 64 G x^3 + \mathcal{O}(x^4).
\end{equation}

\begin{proof}
Apply Thm. \ref{teor1}, with the condition (\ref{thin-contour}), to the WZ-pair
\begin{align}
F(n,k)&=B(n,k) \frac{-64n^3(2k+1)}{(2n-2k-1)^3}, \nonumber \\
G(n,k)&=B(n,k) \frac{(2n+1)^3(42n+5)+k(16k^3-96n^2k-96kn-24k)}{16(2n+1)^3}, \nonumber
\end{align}
where
\[
B(n,k)=\frac { \left( \frac12 - k \right)_n^3 \left( \frac12 + k \right)_n^3}
{(1)_n^3 \left( \frac12 \right)_n^3} (-1)^k \frac{1}{64^n}.
\]
We see that the function $s(x)=s_1(x)+s_2(x)$, where
\[
s_1(x)=\frac{1}{16} \sum_{n=0}^{\infty} \frac{1}{2^{6(n+x)}} \frac{\left(\frac{1}{2}\right)_{n+x}^3}{(1)_{n+x}^3}[42(n+x)+5],
\]
and
\[
s_2(x)=-\frac{1}{64^x} \frac{\left(\frac{1}{2} \right)_x^3} {(1)_x^3}
\frac{64x^3}{(1-2x)^3} \sum_{n=0}^{\infty} \frac{\left( \frac{1}{2}+x \right)_n^3}{\left( \frac{3}{2}-x \right)_n^3} (-1)^n(2n+1)
\]
has the property $s(x+1)=s(x)$.  Then we guess that $s(x)$ is equal to the simple function
\[ t(x)=\frac{1}{\pi} \frac{\cos 3 \pi x}{\cos^3 \pi x}. \]
To prove this result we define the function $h(x)=s(x)-t(x)$. Observe that $h(0)=0$ and $h(x+1)=h(x)$.
Then, use (\ref{asymp2}) to evaluate the limit of $h(x)$ when $x \to 1/2$. We see that it is finite. Hence $h(x)$ has no poles in the band $0 \leq {\rm Re}(x) <1$, and this implies that $h(x)$ is holomorphic. Then we follow the same steps as in our previous proof.
\end{proof}

To prove Identity $3$ we need the following lemma
\begin{lemma}
As $x \to 1/2$ we have
\begin{equation}\label{lemma3}
\sum_{n=0}^{\infty} \frac{\left( \frac12 + 2x \right)_n\left( \frac12 + x \right)_n^2}{\left( \frac{3}{2}-x \right)_n\left( \frac{1}{2} \right)_n^2} \frac{(-1)^n(2n+1+x)}{(2n+1)^2}=\frac{1}{2}-(2x-1)+\mathcal{O}(2x-1)^2.
\end{equation}
\end{lemma}
\begin{proof}
Symbolic computations with Maple yield:
\begin{multline}\nonumber
\sum_{n=0}^{\infty} \frac{\left( \frac12 + 2x \right)_n\left( \frac12 + x \right)_n^2}{\left( \frac{3}{2}-x \right)_n\left( \frac{1}{2} \right)_n^2} \frac{(-1)^n(2n+1+x)}{(2n+1)^2}=
\frac12 \sum_{n=0}^{\infty} \frac{(1)_n}{\left( \frac12 \right)_n}\frac{4n+3}{2n+1} z^n \\ +
\frac12 \sum_{n=0}^{\infty} \frac{(1)_n}{\left( \frac12 \right)_n}\frac{1}{2n+1}z^n
\left[ \frac{}{} \! \! -(10+16n)+(3+4n)H_n+(12+16n)H_{2n+1}  \right].
\end{multline}
Now use the integral representations
\[
H_n=\int_0^1 \frac{1-(1-x)^n}{x}dx, \qquad  \frac{(1)_n}{(2n+1)\left(\frac12\right)_n}=\int_0^1 4^n x^n (1-x)^n dx,
\]
and take the limit as $z \to -1$.
\end{proof}

\subsection*{Identity 3}

\begin{multline}\label{iden3}
\frac{1}{8} \sum_{n=0}^{\infty} \frac{(-1)^n}{2^{2(n+x)}}
\frac{\left(\frac12 \right)_{n+x}\left(\frac14 \right)_{n+x}\left(\frac34 \right)_{n+x}}{(1)_{n+x}^3}[20(n+x)+3] \\
=\frac{1}{\pi} \frac{\cos 2 \pi x}{\cos \pi x} + \frac{1}{4^x} \frac{\left(\frac12 \right)_x\left(\frac14 \right)_x\left(\frac34 \right)_x} {(1)_x^3}
\frac{32x^3}{1-2x} \sum_{n=0}^{\infty} \frac{\left( \frac12 + 2x \right)_n\left( \frac12 + x \right)_n^2}{\left( \frac{3}{2}-x \right)_n\left( \frac{1}{2} \right)_n^2} \frac{(-1)^n(2n+1+x)}{(2n+1)^2}.
\end{multline}
From (\ref{iden3}), we obtain the expansion
\begin{equation}
\frac{1}{8} \sum_{n=0}^{\infty} \frac{(-1)^n}{2^{2(n+x)}}
\frac{\left(\frac12 \right)_{n+x}\left(\frac14 \right)_{n+x}\left(\frac34 \right)_{n+x}}{(1)_{n+x}^3}[20(n+x)+3]=
\frac{1}{\pi}-\frac{3}{2} \pi x^2 + 32 G x^3 + \mathcal{O}(x^4).
\end{equation}

\begin{proof}
We consider the WZ-pair
\begin{align}
F(n,k)&=B(n,k) \frac{-32n^3(n+2k+1)}{(2n-2k-1)(2k+1)^2}, \nonumber \\
G(n,k)&=B(n,k) \frac{(2n+1)^2(20n+3)-k(8n^2+32nk+8k^2+12k-2)}{8(2n+1)^2}, \nonumber
\end{align}
where
\[
B(n,k)=\frac { \left( \frac12 - k \right)_n \left( \frac12 + k \right)_n^2
\left( \frac14 + \frac{k}{2} \right)_n \left( \frac34 + \frac{k}{2} \right)_n}
{(1)_n^3 \left( \frac12 \right)_n^2} (-1)^k (-1)^n \frac{1}{4^n}.
\]
Apply Thm. \ref{teor2} with the condition (\ref{thin-contour}) to the WZ-pair. We see that the function $s(x)=s_1(x)+s_2(x)$, where
\[
s_1(x)=\frac{1}{8} \sum_{n=0}^{\infty} \frac{(-1)^n}{4^{n+x}}
\frac{\left(\frac12 \right)_{n+x}\left(\frac14\right)_{n+x}\left(\frac34 \right)_{n+x}}{(1)_{n+x}^3}[20(n+x)+3],
\]
and
\[
s_2(x)=- \frac{1}{4^x}  \frac{\left(\frac12 \right)_x\left(\frac14 \right)_x\left(\frac34 \right)_x} {(1)_x^3} \frac{32x^3}{1-2x}
\sum_{n=0}^{\infty} \frac{\left( \frac12 + 2x \right)_n\left( \frac12 + x \right)_n^2}{\left( \frac{3}{2}-x \right)_n\left( \frac{1}{2} \right)_n^2} \frac{(-1)^n(2n+1+x)}{(2n+1)^2}
\]
has the property $s(x+1)=-s(x)$. Then we guess that $s(x)$ is equal to
\[ t(x)=\frac{1}{\pi} \frac{\cos 2 \pi x}{\cos \pi x}. \]
To verify this observation, we define the function
\[ h(x)=\frac{s(x)-t(x)}{\cos \pi x} \]
and observe that $h(0)=0$ and $h(x+1)=h(x)$. We use (\ref{lemma3}), and the evaluation at $x=1/2$ from the first series in (\ref{iden3}) (see \cite[Sect. 2.4]{Gu-hyperiden}), to take the limit of $h(x)$ as $x \to 1/2$. The limit is finite, and this implies that $h(x)$ has no poles in the band $0 \leq {\rm Re}(x) < 1$. We conclude that it is holomorphic. When $|{\rm Im}(x)| \to \infty$ we have $|s_1(x)| \to 0$. On the other hand, numerical calculations clearly indicate that for $|{\rm Im}(x)|$ sufficiently large, we have $|s_2(x)|=\mathcal{O}(e^{\pi |{\rm Im}(x)|})$. Hence, we deduce that $|h(x)|=\mathcal{O}(1)$. We complete the proof as usual.
\end{proof}

We need the following lemma to prove identity $4$ :
\begin{lemma}
As $x \to 1/2$, we have the expansion
\begin{equation}\label{lemafinal}
\sum_{n=0}^{\infty} \frac{\left( \frac{1}{2} \right)_n^3}{\left( \frac{3}{2}-x \right)_n^3} (-1)^n(2n+1-x)=
\frac{1}{\pi}-\frac{3 \ln 2}{\pi}(2x-1)+\mathcal{O}(2x-1)^2.
\end{equation}
\end{lemma}
\begin{proof}
First write the formula \cite[Cor. 2.4]{chu} in the form
\begin{equation}\label{forchu}
\frac{1}{2} \sum_{n=0}^{\infty} (-1)^n \frac{\left( \frac12 +x \right)_n^3}{(1)_n^3}(1 + 2x + 4n)=\frac{\cos \pi x}{\pi},
\end{equation}
and then subtracting the formula \cite[Sect. 2.2]{Gu-thesis}
\begin{equation}\label{forGuthesis}
\frac12 \sum_{n=0}^{\infty} (-1)^n \frac{\left( \frac12 +x \right)_n^3}{(1+x)_n^3} (1 + 4x + 4n)= \frac{(1)_x^3}{\left(\frac12 \right)_x^3}\left( \frac{1}{\pi} -\frac{\pi^2}{2}x^2 \right) +\mathcal{O}(x^3),
\end{equation}
and then expand everything in powers of $x$, to see that
\begin{equation}\label{harmoniciden}
\sum_{n=0}^{\infty} (-1)^n \frac{\left( \frac12 \right)_n^3}{(1)_n^3}[-2+3(4n+1)H_n]=-\frac{12 \ln 2}{\pi}.
\end{equation}
When $x \to 1/2$ we obtain the expansion
\begin{multline}\nonumber
\sum_{n=0}^{\infty} (-1)^n \frac{\left( \frac{1}{2} \right)_n^3}{\left( \frac32 - x \right)_n^3}(2n+1-x) =
\frac12 \sum_{n=0}^{\infty} (-1)^n \frac{\left( \frac12 \right)_n^3}{(1)_n^3} (4n+1) \nonumber \\ +
\left( \frac14 \sum_{n=0}^{\infty} (-1)^n \frac{\left( \frac12 \right)_n^3}{(1)_n^3}[-2+3(4n+1)H_n] \right) (2x-1)+\mathcal{O}(2x-1)^2. \nonumber
\end{multline}
Finally use (\ref{forGuthesis}) and (\ref{harmoniciden}) when $x=0$.
\end{proof}

\subsection*{Identity 4}

\begin{multline}\label{iden4}
\frac{1}{2} \sum_{n=0}^{\infty} (-1)^n
\frac{\left(\frac{1}{2}\right)_{n+x}^3}{(1)_{n+x}^3}[4(n+x)+1] \\
=\frac{1}{\pi} \frac{\cos 2 \pi x}{\cos^3 \pi x} +  \frac{\left(\frac{1}{2} \right)_x^3} {(1)_x^3}
\frac{8x^3}{(1-2x)^3} \sum_{n=0}^{\infty} \frac{\left( \frac{1}{2} \right)_n^3}{\left( \frac{3}{2}-x \right)_n^3} (-1)^n(2n+1-x).
\end{multline}
From (\ref{iden4}), we obtain the expansion
\begin{equation}
\frac{1}{2} \sum_{n=0}^{\infty} (-1)^n \frac{\left(\frac{1}{2}\right)_{n+x}^3}{(1)_{n+x}^3}[4(n+x)+1]=\frac{1}{\pi}-\frac{1}{2} \pi x^2 + 8 G x^3 + \mathcal{O}(x^4).
\end{equation}
\begin{proof}
Consider the following WZ-pair:
\begin{align}
F(n,k)&=\frac {\left( \frac12 - k \right)_n^3}
{(1)_n^3} (-1)^k (-1)^n \frac{16n^3(n-2k-1)}{(2n-2k-1)^3}, \nonumber \\
G(n,k)&=\frac {\left( \frac12 - k \right)_n^3}
{(1)_n^3} (-1)^k (-1)^n (4n-2k+1). \nonumber
\end{align}
Observe that the condition (\ref{thin-contour}) in this case is equivalent to saying that
\[
\int_{\mathcal{C}} \widetilde{G}(t,s)ds=8t^3 \frac{\left( \frac12 \right)_t^3}{(1)_t^3}
\int_\mathcal{C} \frac{\left( \frac12 \right)_s^3}{\left( \frac12 -t  \right)_s^3}
(-2s+4t-1) \Gamma(s+1) \Gamma(-s)ds=0,
\]
holds. The formula holds because the only pole inside $\mathcal{C}$ is at $s=0$ and its residue is zero. Then, applying Thm. \ref{teor2} to the WZ-pair, we see that the function $s(x)=s_1(x)+s_2(x)$, where
\[
s_1(x)=\frac{1}{2} \sum_{n=0}^{\infty} (-1)^n
\frac{\left(\frac{1}{2}\right)_{n+x}^3}{(1)_{n+x}^3}[4(n+x)+1],
\]
and
\[
s_2(x)=-\frac{\left( \frac12 \right)_x^3}{(1)_x^3} \frac{8x^3}{(1-2x)^3} \sum_{n=0}^{\infty} \frac{\left( \frac{1}{2} \right)_n^3}{\left( \frac{3}{2}-x \right)_n^3} (-1)^n(2n+1-x),
\]
has the property $s(x+1)=-s(x)$. Then we guess that $s(x)$ is equal to the simple function
\[ t(x)=\frac{1}{\pi}\frac{\cos 2 \pi x}{\cos^3 \pi x}. \]
To prove it, define the function
\[ h(x)=\cos \pi x (s(x)-t(x)). \]
We know that $h(0)=0$ and that $h(x+1)=h(x)$. We use (\ref{lemafinal}) to calculate $h(x)$ when $x \to 1/2$, and we see that it is finite. Hence $h(x)$ has no poles in the band $0 \leq {\rm Re}(x) <1$ and, because $h(x)$ has period $1$, the function is holomorphic.
We easily deduce that $|s_2(x)|$ tends to $0$ as $|{\rm Im}(x)| \to \infty$. By numerical calculations (we do not have a rigorous proof) $|s_1(x)| \to 0$ when $|{\rm Im}(x)| \to \infty$. Hence $|h(x)|=\mathcal{O}(e^{\pi |{\rm Im}(x)|})$, and in the usual way, we deduce that $h(x)=0$ for all complex $x$.
\end{proof}

\section{Identities related to Ramanujan-like series for $1/\pi^2$}\label{section-pi2}

In this section we give incomplete proofs of some new identities related to Ramanujan-like series for $1/\pi^2$. These identities imply expansions in powers of $x$ which agree with the general kind of expansion conjectured in \cite{Gu-matrix} up to order $x^4$. See \cite{AlGu} to understand the relation of these expansions with the theory of Calabi-Yau differential equations. In addition, the hypergeometric identities we are going to prove, allows us to obtain the coefficient of $x^5$.

\subsection{A new WZ-pair and a new identity}

In \cite[Sect. 3]{Gu-hyperiden} we proved two hypergeometric identities related to the series
\[
 \sum_{n=0}^{\infty} \frac{\left( \frac12 \right)_n^3 \left( \frac14 \right)_n \left( \frac34 \right)_n}{
(1)_n^5}(120n^2+34n+3) \frac{1}{2^{4n}}=\frac{32}{\pi^2}.
\]
Here we prove another one which is stronger than those in \cite{Gu-hyperiden}, in the sense of Section \ref{introduc}. The proof begins with the observation that Zeilberger's Maple package \textit{EKHAD} \cite[Appendix]{Pe} certifies that the function
\[ F(n,k)=B(n,k)  \frac{-n^5(2k+1)}{(1+2k-2n)^4}, \]
where
\[
B(n,k)=\frac{ \left( \frac12 - k \right)_n^4 \left( \frac12 + k \right)_n^4
\left( \frac14 \right)_n \left( \frac34 \right)_n}
{(1)_n^5 \left( \frac12 \right)_n^5} \frac{1}{16^n},
\]
has a companion $G(n,k)$ such that $(F,G)$ is a WZ-pair. Applying Thm. \ref{teor1}, we obtain the identity
\begin{multline}\label{idenpi2}
\frac{1}{32} \sum_{n=0}^{\infty} \frac{1}{2^{4(n+x)}}
\frac{\left(\frac12\right)_{n+x}^3\left(\frac14\right)_{n+x}\left(\frac34\right)_{n+x}}{(1)_{n+x}^5}
[120(n+x)^2+34(n+x)+3] \\
=\frac{1}{\pi^2} \frac{8 \cos^4\pi x-12 \cos^2\pi x+5}{2\cos^4\pi x - \cos^2\pi x} - \frac{1}{16^x}
\frac{\left(\frac12 \right)_x^3\left(\frac14 \right)_x\left(\frac34 \right)_x} {(1)_x^5}
\frac{256x^5}{(1-2x)^4} \sum_{n=0}^{\infty} \frac{\left(\frac12+x \right)_n^4}{\left(\frac32-x \right)_n^4}(2n+1) \\
=\frac{1}{\pi^2}-x^2+\frac{10}{3}\pi^2x^4-224 \zeta(3) x^5 + \mathcal{O}(x^6).
\end{multline}

\subsection{Other new identities}

In \cite{Gu-wznew} used the WZ-method to prove the following series for $1/\pi^2$:
\[
 \sum_{n=0}^{\infty} \frac{\left( \frac12 \right)_n^3 \left( \frac13 \right)_n \left( \frac23 \right)_n}{
(1)_n^5}(74n^2+27n+3) \left( \frac{3}{4} \right)^{3n}=\frac{48}{\pi^2}.
\]
A different WZ-proof gives
\begin{multline}\nonumber
\sum_{n=0}^{\infty} \frac{ \left( \frac12 \right)_n^3 \left( \frac13 + \frac{k}{3} \right)_n \left( \frac23 + \frac{k}{3} \right)_n  \left( 1 + \frac{k}{3} \right)_n}{(1)_n^3 (1+k)_n^3} \frac{\left( \frac12 \right)_k^2}{(1)_k^2} \left( \frac{3}{4} \right)^{3n} \\
\frac{n(74n^2+27n+3)+k(108n^2+42nk+24n+5k+1)}{n+\frac{k}{3}}=\frac{48}{\pi^2}.
\end{multline}
Since $F(n,k)=B(n,k)\cdot 384n^3/(3n+k)$, we observe that the power of $n$ is not $5$ and the hypergeometric identity which we can derive using Th. \ref{teor1} is not suitable to easily get an expansion up to order $5$. However we have also found the following new WZ-pair:
\[ F(n,k)=B(n,k)  \frac{2^{12} n^5}{(2n-2k-1)^3}, \]
and
\[ G(n,k)=B(n,k) \left( 74n^2+27n+3+48k^4 \frac{3n+1}{(2n+1)^3}-24k^2 \frac{5n+1}{2n+1} \right) \]
where
\[
B(n,k)=\frac{ \left( \frac12 - k \right)_n^3 \left( \frac12 + k \right)_n^3 \left( \frac13 \right)_n \left( \frac23 \right)_n}
{(1)_n^5 \left( \frac12 \right)_n^3} \left( \frac{3}{4} \right)^{3n}.
\]
By the WZ-method we can prove that the sum of the series
\[
\sum_{n=0}^{\infty} \frac{ \left( \frac12 - k \right)_n^3 \left( \frac12 + k \right)_n^3 \left( \frac13 \right)_n \left( \frac23 \right)_n} {(1)_n^5 \left( \frac12 \right)_n^3} \! \left( \frac{3}{4} \right)^{3n} \! \! \left( 74n^2+27n+3+48k^4 \frac{3n+1}{(2n+1)^3}-24k^2 \frac{5n+1}{2n+1} \right)
\]
is equal to $(\cos^2 \pi k)/\pi^2$. Even more interestingly, applying Th. \ref{teor1} we obtain the hypergeometric identity
\begin{multline}\label{idenpi2}
\frac{1}{48} \sum_{n=0}^{\infty} \left( \frac{3}{4} \right)^{3(n+x)} \frac{\left(\frac12\right)_{n+x}^3 \left(\frac13\right)_{n+x}\left(\frac23\right)_{n+x}}{(1)_{n+x}^5} [74(n+x)^2+27(n+x)+3] \\
=\frac{1}{\pi^2} \frac{8 \cos^4\pi x-8 \cos^2\pi x+3}{4\cos^4\pi x - \cos^2\pi x} + \left( \frac{3}{4} \right)^{3x}
\frac{\left(\frac12 \right)_x^3\left(\frac13 \right)_x\left(\frac23 \right)_x} {(1)_x^5}
\frac{128x^5}{3(2x-1)^3} \sum_{n=0}^{\infty} \frac{\left(\frac12+x \right)_n^3}{\left(\frac32-x \right)_n^3} \\
=\frac{1}{\pi^2}-\frac{1}{3}x^2+\frac{2}{3}\pi^2x^4-\frac{112}{3} \zeta(3) x^5 + \mathcal{O}(x^6).
\end{multline}
Finally, we explain how we guessed the function
\begin{equation}\label{funt}
t(x)=\frac{1}{\pi^2} \frac{8 \cos^4\pi x-8 \cos^2\pi x+3}{4\cos^4\pi x - \cos^2\pi x}.
\end{equation}
First, define the function
\begin{multline}\nonumber
h(x)=\frac{1}{48} \sum_{n=0}^{\infty} \left( \frac{3}{4} \right)^{3(n+x)} \frac{\left(\frac12\right)_{n+x}^3 \left(\frac13\right)_{n+x}\left(\frac23\right)_{n+x}}{(1)_{n+x}^5} [74(n+x)^2+27(n+x)+3] \\ - \left( \frac{3}{4} \right)^{3x} \frac{\left(\frac12 \right)_x^3\left(\frac13 \right)_x\left(\frac23 \right)_x} {(1)_x^5}
\frac{128x^5}{3(2x-1)^3} \sum_{n=0}^{\infty} \frac{\left(\frac12+x \right)_n^3}{\left(\frac32-x \right)_n^3}.
\end{multline}
If we look carefully at the poles of this function, it is possible to guess that the denominator of $t(x)$ is equal to $\cos^2 \pi x (4\cos^2\pi x - 1)$. Thus consider
\[ y(x)=\pi^2 h(x)(4\cos^4\pi x - \cos^2\pi x) \]
and suppose that $y(x)$ is of the form
\[ y(x)=\alpha_1+\alpha_2 \cos^2 \pi x + \alpha_3 \cos^4 \pi x. \]
Making the substitution $x=\pi^{-1} \arccos q$ and evaluating $y(x)$ at several rational values of $q$, gives a simple linear system of equations involving only rational numbers, from which we can determine $\alpha_1$, $\alpha_2$, $\alpha_3$.

\subsection*{Acknowledgment}

To the anonymous referee for his valuable comments and criticism, to Mathew Rogers for reading the revision of the paper improving substantially the style of the English language, and to Arne Meurman for his comments concerning Th. \ref{Fou}.

\end{document}